\documentclass[11pt]{article}
\usepackage[letterpaper,top=2cm,bottom=2cm,left=2.4cm,right=2.4cm,marginparwidth=1.75cm]{geometry}

\usepackage{csquotes}
\usepackage[english]{babel}

\usepackage{dsfont}
\usepackage{amsfonts}
\usepackage{amsmath, amssymb}
\usepackage{amsthm}
\usepackage{mathrsfs}
\usepackage{hyperref}
\usepackage{graphicx}
\usepackage{float} 
\usepackage[]{caption}
\usepackage[position=b]{subcaption}
\usepackage{titlesec}
\usepackage[usenames, dvipsnames]{color}
\usepackage[svgnames]{xcolor}
\usepackage{mdframed}
\usepackage[ddmmyyyy]{datetime}
\usepackage{enumitem}
\usepackage{placeins} 
\usepackage{indentfirst}
\usepackage{cancel}

\let\Oldsection\section
\renewcommand{\section}{\FloatBarrier\Oldsection}
\let\Oldsubsection\subsection
\renewcommand{\subsection}{\FloatBarrier\Oldsubsection}
\let\Oldsubsubsection\subsubsection
\renewcommand{\subsubsection}{\FloatBarrier\Oldsubsubsection}
\renewcommand{\P}{\mathbb P}

\titleformat{\section}{\bfseries\Large }{\thesection.}{0.5em}{} 
\titleformat{\subsection}{\large\bfseries}{\thesubsection.}{0.5em}{}
\titleformat{\subsubsection}{\bfseries}{\thesubsubsection.}{0.5em}{}
\titleformat{\title}{\bfseries\huge}{\thetitle}{0.5em}{}

\hypersetup{      
    colorlinks=true,   
    breaklinks=true,   
    urlcolor= blue,    
    linkcolor= black,  
    citecolor= blue   
    }

\theoremstyle{plain}

\newtheorem{theorem}{Theorem}
\newtheorem{proposition}[theorem]{Proposition}

\newtheorem{lemma}{Lemma}[theorem]
\theoremstyle{definition}

\theoremstyle{remark}

\surroundwithmdframed[linewidth = 2.5pt, topline=false, rightline=false, bottomline=false, innertopmargin= -10pt, leftmargin=0, skipabove = 0pt, skipbelow=5pt]{theorem}
\surroundwithmdframed[linewidth = 2.5pt, topline=false, rightline=false, bottomline=false, innertopmargin= -10pt, leftmargin=0, skipabove = 0pt, skipbelow=\medskipamount]{proposition}
\surroundwithmdframed[linewidth = 1pt, topline=false, rightline=false, bottomline=false, innertopmargin= -1pt, leftmargin=0, skipabove = 5pt, skipbelow=5pt]{proof}

\definecolor{mygray}{gray}{0.3}
\definecolor{bluegray}{RGB}{100, 120, 200}

\surroundwithmdframed[linewidth = 1pt, linecolor=bluegray, rightline=false, leftmargin=0.15\linewidth, rightmargin=0.1\linewidth, skipabove = 5pt, skipbelow=0pt]{sidenote}

\renewcommand{\bf}[1]{\textbf{#1}}

\newcommand{\proba}[2]{\mathbb{P}\left( #2 \right)}
\newcommand{\expect}[2]{\mathbb{E}_{#1}\left[ #2 \right]}

\renewcommand{\epsilon}{\varepsilon}

\newcommand{\T}{\mathcal{T}}

\newcommand{\node}[1]{\raisebox{.5pt}{\textcircled{\raisebox{-.9pt} {#1}}}}
\newcommand{\enode}[1]{\emph{\raisebox{.5pt}{\textcircled{\raisebox{-.9pt} {#1}}}}}

\title{\textsc{Eve, Adam and the Preferential Attachment Tree}}
\author{Alice Contat\thanks{Universit\'e Paris-Saclay.\hfill  \href{mailto:alice.contat@universite-paris-saclay.fr}{\texttt{alice.contat@universite-paris-saclay.fr}}}, Nicolas Curien\thanks{Universit\'e Paris-Saclay.\hfill  \href{mailto:nicolas.curien@gmail.com}{\texttt{nicolas.curien@gmail.com}}}, Perrine Lacroix\thanks{Universit\'e Paris-Saclay.\hfill  \href{mailto:perrine.lacroix@universite-paris-saclay.fr}{\texttt{perrine.lacroix@universite-paris-saclay.fr}}}, Etienne Lasalle\thanks{Universit\'e Paris-Saclay.\hfill  \href{mailto:etienne.lasalle@universite-paris-saclay.fr}{\texttt{etienne.lasalle@universite-paris-saclay.fr}}}, Vincent Rivoirard\thanks{Universit\'e Paris-Dauphine.\hfill  \href{mailto:vincent.rivoirard@dauphine.fr}{\texttt{vincent.rivoirard@dauphine.fr}}}}
\date{}

\begin{document}

\maketitle
\begin{abstract}\noindent We consider the problem of finding the initial vertex (Adam) in a Barab\'asi--Albert tree process  $ (\mathcal{T}(n) : n \geq 1)$ at large times. More precisely, given $ \varepsilon>0$, one wants to output a subset $ \mathcal{P}_{ \varepsilon}(n)$ of  vertices of  $ \mathcal{T}(n)$ so that the initial vertex belongs  to $ \mathcal{P}_ \varepsilon(n)$ with probability at least $1- \varepsilon$ when $n$ is large. It has been shown by Bubeck, Devroye \& Lugosi \cite{bubeck2017finding}, refined later  by Banerjee \& Huang \cite{banerjee2021degree},  that one needs to output at least $ \varepsilon^{-1 + o(1)}$  and at most $\varepsilon^{-2 + o(1)}$ vertices to succeed. We prove that the exponent in the lower bound is sharp and the key idea is that Adam is either a ``large degree" vertex or is a neighbor of a ``large degree" vertex (Eve). \end{abstract}


\begin{figure}[!h]
 \begin{center}
 \includegraphics[width=10.6cm]{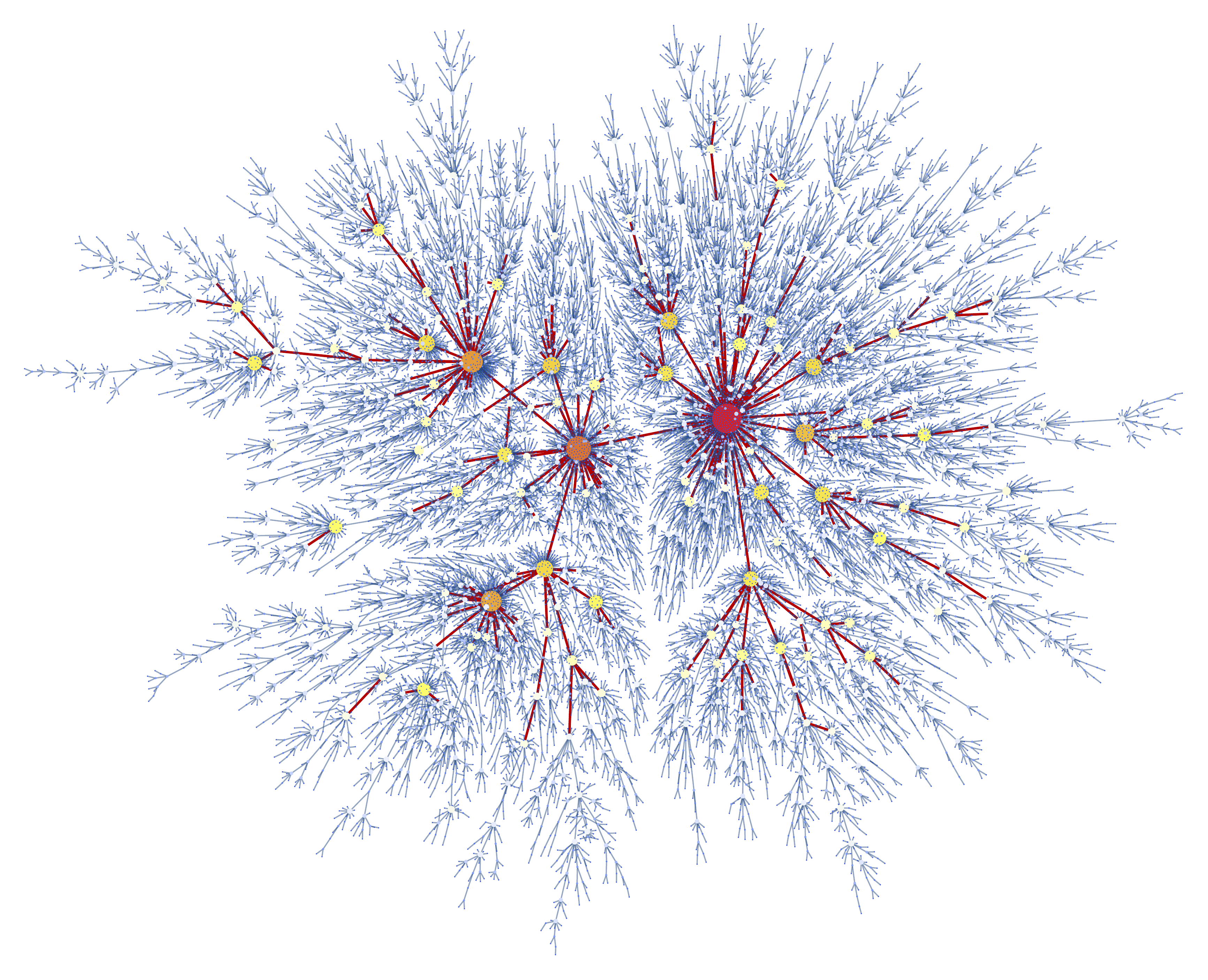}\\

 \caption{A simulation of $ \mathcal{T}(20000)$ where the degree of vertices are represented by disks of increasing sizes and varying colors. Given the tree, the vertices adjacent to the thick red edges form our output subset of vertices which contains Adam (and Eve) with large probability. Notice that this set is not necessarily connected.}
 \end{center}
 \end{figure}
\section*{Introduction}
We consider the classical Barab\'asi--Albert growing tree process \cite{barabasi1999emergence}: We start from $ \mathcal{T}(1)$ the tree with one single vertex $\raisebox{.5pt}{\textcircled{\raisebox{-.9pt} {1}}}$ labeled $1$, at step $2$ the vertex $ \raisebox{.5pt}{\textcircled{\raisebox{-.9pt} {2}}}$ labeled $2$ attaches to $1$ to form $ \mathcal{T}(2)$, and for $n \geq 3$ the tree  $\T(n)$ is obtained from $\T(n-1)$ inductively, by attaching the vertex $ \raisebox{.5pt}{\textcircled{\raisebox{-.9pt} {$n$}}}$ to a random vertex of label $1 \leq i \leq n-1$ with probability proportional to $d_i(n)$, the degree of node labeled $i$ in $\T(n)$. This process is perhaps the most popular random graph constructed by \emph{preferential attachment} rules.  They have received a lot of attention in the last decades since they can serve as sensible models for many real-world networks, in particular due to their  remarkable \emph{scale-free property} (i.e.~asymptotic degree distribution with a power law tail). The literature on the subject is
extremely vast, and we refer to \cite{van2009random} for  background and references. In this note, we consider the so-called \emph{network archeology}, see the seminal reference \cite{bubeck2017finding}.

\paragraph{Network archeology and finding Adam.} Assume that we are given the random tree $\mathcal{T}(n)$ without its labeling. The goal is to recover, as an archeologist, the initial vertex  $\raisebox{.5pt}{\textcircled{\raisebox{-.9pt} {1}}}$, called Adam.  Of course this is not possible for sure, since, for example, the vertices $\raisebox{.5pt}{\textcircled{\raisebox{-.9pt} {1}}}$ and $ \raisebox{.5pt}{\textcircled{\raisebox{-.9pt} {2}}}$ are exchangeable in $ \mathcal{T}(n)$. Therefore, given a threshold $  \varepsilon>0$, our goal is to output $\mathcal{P}_ { \varepsilon}(n)$, a subset of vertices of $\mathcal{T}(n)$ depending only on the tree structure (as labels are not available), as small as possible, so that   \begin{eqnarray} \label{eq:adam} \liminf_{n \to \infty} \mathbb{P}( \raisebox{.5pt}{\textcircled{\raisebox{-.9pt} {1}}} \in \mathcal{P}_ \varepsilon(n)) \geq 1- \varepsilon.  \end{eqnarray} Various recipes have been developed in the literature mainly based on the notion of \emph{degree}, \emph{Jordan centrality} or \emph{product centrality} of a vertex. They yield subsets $ \mathcal{P}_{ \varepsilon}(n)$ whose size does not depend on $n$ (asymptotically). 
In the case when $( \mathcal{T}(n) : n \geq 1)$ is a uniform random recursive tree process (i.e.~when the vertex labeled $n+1$ connects to a uniform vertex among $\{ \raisebox{.5pt}{\textcircled{\raisebox{-.9pt} {1}}} , \cdots, \raisebox{.5pt}{\textcircled{\raisebox{-.9pt} {$n$}}}\}$), it is known that one can take $ \mathcal{P}_ \varepsilon(n)$ whose size $| \mathcal{P}_{ \varepsilon}(n)|$  converges as $n \to \infty$  and grows as $ \varepsilon \downarrow 0$, see \cite{bubeck2017finding}. In the Barab\'asi--Albert model, by taking the $k$ nodes of largest degree in $ \mathcal{T}(n)$ it is proved in \cite{bubeck2017finding} and refined in \cite{banerjee2021degree} that one can construct $ \mathcal{P}_{ \varepsilon}(n)$ satisfying \eqref{eq:adam} so that 
 \begin{eqnarray} \label{eq:upperbounddegree} |\mathcal{P}_{ \varepsilon}(n)| \leq \varepsilon^{-2+ o(1)},  \end{eqnarray} as $ \varepsilon \to 0$.  On the other hand, a lower bound is proved in \cite{bubeck2017finding} and shows that any subset $\mathcal{P}_{ \varepsilon}(n)$ of vertices of $\mathcal{T}(n)$ which does not depend on the labeling of  $ \mathcal{T}(n)$ and  satisfies \eqref{eq:adam}, must obey
$$ |\mathcal{P}_{ \varepsilon}(n)| \geq \varepsilon^{-1+ o(1)},$$
as $ \varepsilon \to 0$. In this note, we fill this gap and prove that the exponent in the lower bound is sharp:

\begin{theorem}\label{theo:main_result}
 Let $ \eta \in (0,1/8)$ be fixed. For any $ \varepsilon \in (0,1)$, one can construct subsets $ \mathcal{P}_{  \varepsilon}(n)$ of vertices of $\mathcal{T}(n)$ which only depends on the tree structure of $ \mathcal{T}(n)$ (not on its labeling)  so that for all $ \varepsilon$ small enough, 
  \begin{eqnarray*} \mathbb{P} \left( \sup_{n \geq 1} |\mathcal{P}_{ \varepsilon}(n)| \leq \varepsilon^{-1- \eta} \quad \mbox{and} \quad   \emph{\raisebox{.5pt}{\textcircled{\raisebox{-.9pt} {1}}}} \in  \mathcal{P}_{ \varepsilon}(n)  \mbox{ for all } n \geq 1\right) \geq 1 - \varepsilon^{1- \eta} . \end{eqnarray*} 
\end{theorem}
\begin{samepage}Our strategy to construct $ \mathcal{P}_{ \varepsilon}(n)$ can be summarized in one sentence: \begin{center}\emph{we use Eve to find Adam}! \end{center} \end{samepage}
Let us be more precise and introduce a piece of notation.
In order to avoid confusion (especially with time steps), we usually write  $\raisebox{.5pt}{\textcircled{\raisebox{-.9pt} {$i$}}}$ for the vertex $i$, so that the vertex set of $ \mathcal{T}(n)$ is $\{ \raisebox{.5pt}{\textcircled{\raisebox{-.9pt} {$1$}}}, \raisebox{.5pt}{\textcircled{\raisebox{-.9pt} {$2$}}},\dots , \raisebox{.5pt}{\textcircled{\raisebox{-.9pt} {$n$}}}\}$. For $1 \leq i \leq n$, we denote by $d_i(n)$ the degree of  $\raisebox{.5pt}{\textcircled{\raisebox{-.9pt} {$i$}}}$ in $ \mathcal{T}(n)$, the sum of degrees in $ \mathcal{T}(n)$ being equal to $2(n-1)$.  We also introduce the sequence $ (\alpha_n : n \geq 2)$ defined by $\alpha_2=1$ and $$\frac{\alpha_{n+1}}{\alpha_n} = 1 + \frac{1}{2(n-1)},$$ so that $\alpha_n = \prod_{k=1}^{n-2} (1+ \frac{1}{2k} ) \sim \frac{2 \sqrt{n}}{ \sqrt{ \pi}}$ as $n \to \infty$. It follows from the preferential attachment dynamics that $ \frac{d_i(n)}{\alpha_n}$ is a positive martingale (for the canonical filtration) and so for every $i \geq1$,
$\frac{d_i(n)}{ \alpha_n}  $ converges almost surely for each fixed $i$ as $n \to \infty$. To match with the notation of \cite{pekoz2017joint}, which is extensively used in the following, we introduce the renormalized degrees and their limits \begin{eqnarray} \label{eq:defrenormalizeddegrees} D_i(n) := \frac{d_i(n)}{ \alpha_n \cdot \sqrt{\pi}} \quad \mbox{ and } \quad  	D_{i}(n) \xrightarrow[n\to\infty]{a.s.} \mathbf{D}_{i}.  \end{eqnarray} We say that $  \mathbf{D}_i$ is the ``limit" degree of $ \raisebox{.5pt}{\textcircled{\raisebox{-.9pt} {$i$}}}$.  Actually the former point-wise convergence of martingales holds in the stronger $\ell^{\infty}$ sense as shown by M\'ori \cite{mori2005maximum}:    
\begin{eqnarray} \label{eq:degreelimite} (D_{i}(n) : i \geq 1) \xrightarrow[n\to\infty]{a.s.} ( \mathbf{D}_{i} : i \geq 1).\end{eqnarray}
 As mentioned before, large degree vertices  in $ \mathcal{T}(n)$ have already been used to find Adam, see \cite{bubeck2017finding,banerjee2021degree}. More precisely, it has been proved in \cite{banerjee2021degree} that the first $ \varepsilon^{-2+o(1)}$ vertices of largest degree of $ \mathcal{T}(n)$ are roughly made of all the vertices of $ \mathcal{T}(n)$ with renormalized degree $D_i(n)$ larger than $ \varepsilon$. On the other hand, a classical estimate (see \eqref{eq:control_adam_toujours} below) shows that $ \mathbf{D}_1$ is at least $ \varepsilon$ with probability of order $ 1- \varepsilon$ thus entailing \eqref{eq:adam} and \eqref{eq:upperbounddegree}. Our idea is then the following: instead of looking for Adam, we rather consider the pair Adam + Eve, where Eve stands for the vertex $\raisebox{.5pt}{\textcircled{\raisebox{-.9pt} {$2$}}}$. Indeed,  when Adam has a ``small degree", this means that at the first step of the process, the vertex $\raisebox{.5pt}{\textcircled{\raisebox{-.9pt} {$3$}}}$ must have connected to $\raisebox{.5pt}{\textcircled{\raisebox{-.9pt} {$2$}}}$ instead of $\raisebox{.5pt}{\textcircled{\raisebox{-.9pt} {$1$}}}$ and this will force the limit degree of Eve to be large. The subset $ \mathcal{P}_{ \varepsilon}(n)$ is then constructed by gathering the (vertices incident to) edges  of $ \mathcal{T}(n)$ connecting the not-so-small degree vertices of $ \mathcal{T}(n)$ which are connected to large degree vertices. The proper construction is presented below.\medskip

\textbf{Acknowledgments:} We thank Gabor Lugosi for  incentives  and discussions about network archeology. This work has been supported by ANR RanTanPlan and by the ERC Advanced Grant 740943 (GeoBrown). We are indebted to Matthieu Lerasle for pointing us a shortcoming in the first version of this article.


\section{Definition of the packet}
Let $ \varepsilon \in (0,1)$ and $n \geq 1$ be fixed. We set  \begin{equation} \label{eq:packet}
\mathcal{P}_{  \varepsilon}(n) := 
\left \{ \node{i}, \ 1 \leq i \leq n : \ \begin{array}{c} \exists 1 \leq  j \leq n, \\ \node{j} \sim \node{i} \end{array},  \quad \begin{array}{c} D_j(n) \cdot (D_{i}(n))^{2} > \varepsilon \\   \mathrm{or} \\ D_{i}(n) \cdot (D_{j}(n))^{2}> \varepsilon \end{array}\right \}, 
\end{equation}
where $\node{j} \sim \node{i}$ denotes the fact that $\node{j}$ and $\node{i}$ are connected by an edge in $\mathcal{T}(n)$. We shall call this subset the $ \varepsilon$-packet of $ \mathcal{T}(n)$. It is thus made of the set of neighbor vertices $\node{i}$ and $\node{j}$ in $ \mathcal{T}(n)$ such that their degrees are not too small and either $d_i(n)$ or $d_j(n)$ is large. Clearly, the construction of the $ \varepsilon$-packet only depends on the graph structure of $ \mathcal{T}(n)$ and not on its labeling (and even better, it is a local construction). In Section~\ref{sec:into}, we prove that $ \mathcal{P}_{\varepsilon}(n)$ contains Adam (and Eve) with large probability. In Section~\ref{sec:size}, we prove that $ \mathcal{P}_{\varepsilon}(n)$ has the appropriate size. 
\medskip

To help intuition, let us describe the heuristics behind \eqref{eq:packet}: In most cases, Adam is a large degree vertex and in fact its limit degree is larger than $ \varepsilon$ with probability at least $1- \varepsilon$, see \eqref{eq:control_adam_toujours} below. Hence, we shall restrict to vertices with renormalized degree larger than $ \varepsilon^{-1+o(1)}$. But as mentioned in the introduction, there are of order $ \varepsilon^{-2+o(1)}$ vertices satisfying this bound, providing 
the bound~\eqref{eq:upperbounddegree}. 
To improve the latter, observe that when Adam is unusually small, then Eve must be large. Indeed, since the preferential attachment process tends to increase the degree of high degree nodes as $n$ grows, whether $\node{3}$ gets attached to $\node{1}$ or $\node{2}$ greatly impacts the asymptotic behavior of Adam or Eve's degrees. Said differently, the neighbor of $\node{3}$ in $\mathcal{T}(3)$ will  be very unlikely to have an abnormally small limit degree. To fix ideas, imagine that $\node{2} \sim \node{3}$, then we will prove that $  \mathbf{D}_1\mathbf{D}_2^2 > \varepsilon$ with probability at least $1- \varepsilon^{1- \eta}$ thus motivating our definition of the $ \varepsilon$-packet.


\section{Adam is in $ \mathcal{P}_ \varepsilon(n)$ with high probability} \label{sec:into}
Recall the definition of the $ \varepsilon$-packet in \eqref{eq:packet}. The goal of this section is to prove:
\begin{proposition}[Adam is in the $ \varepsilon$-packet]\label{prop:adam_in_the_packet} Let $ \eta\in (0,1)$ be fixed. For $\varepsilon>0$ small enough, we have 
\begin{equation*} \mathbb{P} \left(\emph{\raisebox{.5pt}{\textcircled{\raisebox{-.9pt} {1}}}} \in  \mathcal{P}_{ \varepsilon}(n)  \mbox{ for all } n \geq 1\right) \geq 1 - \varepsilon^{1 - \eta}.
\end{equation*} 
\end{proposition}
By \eqref{eq:degreelimite}, the indices belonging to the packet $ \mathcal{P}_ \varepsilon(n)$ converge as $n \to \infty$ to a random subset of indices and we shall first focus on those. For this purpose, we  examine the distribution of $(\mathbf{D}_{1}, \mathbf{D}_{2})$ which is explicit thanks to \cite{pekoz2017joint}. In the rest of the paper $ \mathrm{Cst}>0$ denotes a constant which does not depend on $\varepsilon, \eta, i$ or $n$, but whose value may change from line to line. 

\subsection{Estimations on the joint degrees of Adam and Eve}
 
For $u>0$ and $v>0$, we recall the definition of the Beta distribution $\beta(u, v)$ and of the generalized Gamma distribution  $\mathrm{GG}(u,v)$ whose densities are given for $x\in{\mathbb R}$ by 
 \begin{eqnarray} \label{eq:densities} f_{\beta(u,v)}(x) = \frac{\Gamma(u+v)}{\Gamma(u) \Gamma(v)} x^{u-1}(1-x)^{v-1} \mathbf{1}_{0<x<1}  \quad \mbox{ and } \quad f_{\mathrm{GG}(u,v)}(x) = \frac{v}{\Gamma(u/v)}x^{u-1} \mathrm{e}^{-x^{v}} \mathbf{1}_{x >0} .  \end{eqnarray}
\begin{lemma}[Special case of Theorem 1.1 from \cite{pekoz2017joint}]
%
The limit joint distribution of the rescaled degrees of $\enode{1}$ and $\enode{2}$, conditionally on $\enode{3} \sim \enode{2}$ is given by 
\begin{equation}
(\mathbf{D}_1, \mathbf{D}_2) | \emph{\node{3}} \sim \emph{\node{2}} \overset{(d)}{=} (B_1 B_2 Z_3, (1-B_1) B_2 Z_3).
\label{eq:limit_dstri_D1_D2}
\end{equation}
where $B_1 \sim \beta(1,2)$, $B_2 \sim \beta(3,1)$, $Z_3 \sim \mathrm{GG} (5,2)$ and  $B_1$, $B_2$ and $Z_3$ are independent.
\label{lem:limit_degree_distribution}
\end{lemma} 
Note that the roles of $\node{1}$ and $\node{2}$ are exchangeable thus, this lemma also provides the distribution of $ \mathbf{D}_1$ and $ \mathbf{D}_2$ conditionally on $\node{3} \sim \node{1}$. 
\begin{lemma}There exists $ \mathrm{Cst}>0$ so that for all $a>b \geq 0$ and $ \varepsilon \in (0,1)$ we have \label{lem:control_of_the_limit_degree}
\begin{eqnarray} 
\proba{}{\mathbf{D}_1 \leq \varepsilon^a \quad \text{and} \quad \mathbf{D}_2 \leq \varepsilon^b \mid  \enode{3} \sim \enode{2} } &\leq& \mathrm{Cst} \cdot \varepsilon^{a+2b} 
\label{eq:control_adam_eve_limit_degree}\\
\proba{}{\mathbf{D}_1 \leq \varepsilon \mid  \enode{3} \sim \enode{2} } &\leq& \mathrm{Cst} \cdot \varepsilon   \label{eq:control_adam_toujours}\\
\proba{}{\mathbf{D}_2 \leq \varepsilon^a \quad \text{and} \quad \mathbf{D}_1 \leq \varepsilon^b \mid  \enode{3} \sim \enode{2} } &\leq& \mathrm{Cst} \cdot \varepsilon^{2a+b} 
\label{eq:control_adam_eve_limit_degree_inverse}\\
\proba{}{\mathbf{D}_2 \leq \sqrt{\varepsilon} \mid  \enode{3} \sim \enode{2} } &\leq& \mathrm{Cst} \cdot \varepsilon.
\label{eq:control_eve_toujours}
\end{eqnarray}

\end{lemma}


\begin{proof}
Using \eqref{eq:limit_dstri_D1_D2}, for $a  > b > -\infty$ we have 
\begin{eqnarray*}
&\ &\proba{}{\mathbf{D}_1 \leq \varepsilon^a \quad \text{and} \quad \mathbf{D}_2 \leq \varepsilon^b \mid  \node{3} \sim \node{2} }\\
& = &  \mathbb{P} \left( B_1B_2Z_3 \leq {\varepsilon}^a \mbox{ and } (1- B_1)B_2Z_3 \leq {\varepsilon}^b\right) \\
& \leq&  \mathbb{P} \left( B_2 Z_3 \leq 2\varepsilon^b \mbox{ and } B_1 \leq \frac{ \varepsilon^a}{B_2 Z_3}\right) + \mathbb{P} \left(B_2Z_3 > 2 \varepsilon^b \mbox{ and } B_1<\frac{1}{2} \mbox{ and } (1-B_1) < \frac{1}{2} \right)\\
&= &  \mathbb{P} \left( B_2 Z_3 \leq 2\varepsilon^b \mbox{ and } B_1 \leq \frac{ \varepsilon^a}{B_2 Z_3}\right).
\end{eqnarray*}
 To compute this probability, we use the fact that $ \mathbb{P}(B_1 \leq x) = 2x - x^2 \leq 2x$. This gives 
 \begin{eqnarray*}
&\ & \int_{0}^{1} \mathrm{d}x \int_{0}^{1} \mathrm{d}y \int_{0}^{ \infty} \mathrm{d}z 2(1-x) \cdot 3 y^2 \cdot \frac{2}{ \Gamma(5/2)} z^4 \mathrm{e}^{-z^2} \mathds{1}_{y z \leq 2\varepsilon^b} \mathds{1}_{x y z \leq \varepsilon^a} \\
&=& \int_{0}^{1} \mathrm{d}y \int_{0}^{ \infty} \mathrm{d}z\ \frac{6}{ \Gamma(5/2)} y^2 z^4 \mathrm{e}^{-z^2} \mathds{1}_{ y z \leq 2\varepsilon^b} \left( \int_{0}^{1} \mathrm{d}x\  2(1-x)\mathds{1}_{x \leq \frac{\varepsilon^a}{yz}}\right) \\
&\leq& \int_{0}^{1} \mathrm{d}y \int_{0}^{ \infty} \mathrm{d}z \  \frac{6}{ \Gamma(5/2)} y^2 z^4 \mathrm{e}^{-z^2} \mathds{1}_{y z \leq 2\varepsilon^b} \cdot 2 \frac{\varepsilon^a}{yz}  \\
&=&  2\varepsilon^a \int_{0}^{1} \mathrm{d}y \int_{0}^{ \infty} \mathrm{d}z \ \frac{6}{ \Gamma(5/2)} y  z^3 \mathrm{e}^{-z^2} \mathds{1}_{y z \leq 2\varepsilon^b}  \\
&=& \mathrm{Cst} \cdot  \varepsilon^{a}\int_{0}^{ \infty} z^3 \mathrm{e}^{-z^2}\mathrm{d}z \  \int_{0}^{1}  y   \mathds{1}_{y z \leq 2\varepsilon^b}\mathrm{d}y\\
&\leq& \mathrm{Cst} \cdot  \varepsilon^{a}\int_{0}^{ \infty} z^3 \mathrm{e}^{-z^2}\mathrm{d}z \times \left(\Bigg(\frac{2\varepsilon^b}{z}\Bigg)^2 \wedge 1 \right) \leq 
 \mathrm{Cst} \cdot  \varepsilon^{a} (1 \wedge \varepsilon^{2b}).
\end{eqnarray*}
The first two lines~\eqref{eq:control_adam_eve_limit_degree} and~\eqref{eq:control_adam_toujours} of the lemma follow.
Similarly, partitioning on the value of $B_2 Z_3$ with respect to $2 \varepsilon^b$, we can obtain 
\begin{equation*}
\proba{}{\mathbf{D}_2 \leq \varepsilon^a \quad \text{and} \quad \mathbf{D}_1 \leq \varepsilon^b \mid  \node{3} \sim \node{2} }
\leq   \mathbb{P} \left( B_2 Z_3 \leq 2\varepsilon^b \mbox{ and } (1-B_1) \leq \frac{ \varepsilon^a}{B_2 Z_3}\right).
\end{equation*}
Using the fact that $ \mathbb{P} (1-B_1 \leq x) = x^2$, the computation of this probability is done along the same lines and gives the last two inequalities.

\end{proof}


\subsection{Adam is in the packet}


We now prove Proposition~\ref{prop:adam_in_the_packet}. Let us first compute the probability that Adam is not eventually in the $ \varepsilon$-packet, that is
\begin{eqnarray*}
 1-\proba{}{\node{1} \in \mathcal{P}_ \varepsilon(n) \mbox{ for all }n \mbox{ large enough}}  &\underset{ \eqref{eq:degreelimite}}{\leq}& \ \proba{}{ \mathbf{D}_1 \mathbf{D}_2^2 \leq \varepsilon \mbox{ and }  \mathbf{D}_1^2  \mathbf{D}_2 \leq \varepsilon }\\
&\underset{ \mathrm{symmetry}}{=} & \ \proba{}{ \mathbf{D}_1  \mathbf{D}_2^2 \leq \varepsilon  \mbox{ and }  \mathbf{D}_1^2  \mathbf{D}_2 \leq \varepsilon  | \node{3} \sim \node{2}} 
\end{eqnarray*}

We now partition depending on the position of $ \mathbf{D}_1$ with respect  to $\varepsilon$ and 1: \begin{eqnarray*}
& & \hspace{-1.5cm}\proba{}{ \mathbf{D}_1  \mathbf{D}_2^2 \leq \varepsilon  \mbox{ and }  \mathbf{D}_1^2  \mathbf{D}_2 \leq \varepsilon  | \node{3} \sim \node{2}} \\
&\leq & \mathbb{P}(\mathbf{D}_1 \leq \varepsilon| \node{3} \sim \node{2})   \\
& &+\  \proba{}{\mathbf{D}_1 \mathbf{D}_2^2 \leq \varepsilon  \mbox{ and } \mathbf{D}_1^2 \mathbf{D}_2 \leq \varepsilon \mbox{ and } \mathbf{D}_1 \geq 1 | \node{3} \sim \node{2}} \\
& &+ \ \proba{}{\mathbf{D}_1 \mathbf{D}_2^2 \leq \varepsilon  \mbox{ and } \mathbf{D}_1^2 \mathbf{D}_2 \leq \varepsilon \mbox{ and }  \varepsilon \leq \mathbf{D}_1 \leq 1 | \node{3} \sim \node{2}} \\
&\leq & \mathbb{P}(\mathbf{D}_1 \leq \varepsilon| \node{3} \sim \node{2}) \\
& &+ \ \mathbb{P}( \mathbf{D}_2 \leq  \sqrt{\varepsilon}| \node{3} \sim \node{2}) \\
& &+ \  \proba{}{\mathbf{D}_1 \mathbf{D}_2^2 \leq \varepsilon  \mbox{ and } \mathbf{D}_1^2 \mathbf{D}_2 \leq \varepsilon  \mbox{ and }  \varepsilon \leq \mathbf{D}_1 \leq 1 | \node{3} \sim \node{2}} . 
\end{eqnarray*}
The first two terms are easily  handled using \eqref{eq:control_adam_toujours} and \eqref{eq:control_eve_toujours} and are $O( \varepsilon)$. Let us focus now on the last term. 
We  decompose the event $\{ \varepsilon \leq  \mathbf{D}_1 \leq1\}$ in a finite union depending on the order of magnitude of $\mathbf{D}_1$, via 
\begin{equation*}
\{ \varepsilon < \mathbf{D}_1 \leq  1 \} \subseteq \bigcup\limits_{a \in \mathcal{A}} \{ \varepsilon^a < \mathbf{D}_1 \leq 2 \varepsilon^a \},
\end{equation*}
where $\mathcal{A}$ is a finite subset of $[0,1]$. This decomposition can be obtained by discretizing the segment $[0, 1]$ with steps of size $\log_2(1/\varepsilon)^{-1}$, so the cardinality of $\mathcal{A}$ satisfies $|\mathcal{A}|\leq \log_2(1/\varepsilon)+1$.  For any $a \in \mathcal{A}$ we have
\begin{equation*}
\left\{ \mathbf{D}_1 \mathbf{D}_2^2 \leq \varepsilon \mbox{ and } \mathbf{D}_1^2 \mathbf{D}_2 \leq \varepsilon  \right\} \subseteq \left\{ \mathbf{D}_1 \leq \varepsilon^a \quad \textnormal{or} \quad \left( \mathbf{D}_2 \leq \varepsilon^{\frac{1-a}{2}} \mbox{ and } \mathbf{D}_2 \leq \varepsilon^{1-2a} \right)  \right\}.
\end{equation*}
Thus, 
\begin{equation*}
\left\{ \mathbf{D}_1 \mathbf{D}_2^2 \leq \varepsilon  \mbox{ and } \mathbf{D}_1^2 \mathbf{D}_2 \leq \varepsilon \right\} \cap \left\{ \varepsilon^a < \mathbf{D}_1 \leq  2 \varepsilon^a \right\} \subseteq \left\{ \mathbf{D}_1 \leq 2 \varepsilon^a  \textnormal{ and }  \mathbf{D}_2 \leq \varepsilon^{\frac{1-a}{2}} \textnormal{and } \mathbf{D}_2 \leq \varepsilon^{1-2a}\right\}.
\end{equation*}
By taking the union over all $a$ in $\mathcal{A}$, we get that
\begin{equation*}
\{ \varepsilon < \mathbf{D}_1 \leq 1 \textnormal{ and }   \mathbf{D}_1 \mathbf{D}_2^2 \leq \varepsilon  \mbox{ and } \mathbf{D}_1^2 \mathbf{D}_2 \leq \varepsilon \} \subseteq \bigcup\limits_{a \in \mathcal{A}} \left\{ \mathbf{D}_1 \leq 2 \varepsilon^a \textnormal{ and }  \mathbf{D}_2 \leq \varepsilon^{\frac{1-a}{2}}  \textnormal{ and }  \mathbf{D}_2 \leq \varepsilon^{1-2a}\right\},
\end{equation*}
which yields, after applying a union bound and (\ref{eq:control_adam_eve_limit_degree}) if $2\varepsilon^a < \varepsilon^{\frac{1-a}{2}}$ or \eqref{eq:control_adam_eve_limit_degree_inverse} otherwise,
\begin{equation*}
 \proba{}{ \varepsilon < \mathbf{D}_1 \leq 1  \textnormal{ and }    \mathbf{D}_1 \mathbf{D}_2^2 \leq \varepsilon \textnormal{ and }  \mathbf{D}_1^2 \mathbf{D}_2 \leq \varepsilon | \node{3} \sim \node{2}} \leq  \mathrm{Cst} \cdot  \ \log_2(1/\varepsilon)  \cdot \varepsilon.
\end{equation*}
Gathering up the pieces, we indeed see that for every $ \eta>0$ we have for $\varepsilon >0$ small enough    \begin{eqnarray} \label{eq:limitfirst}\proba{}{ \mathbf{D}_1  \mathbf{D}_2^2 \leq \varepsilon  \mbox{ and }  \mathbf{D}_1^2  \mathbf{D}_2 \leq \varepsilon  | \node{3} \sim \node{2}} \leq  \varepsilon^{1-\eta}.  \end{eqnarray} Let us now make this uniform in $n$. For this we use the Markov property and a stopping time argument. More precisely, let $$\theta= \inf \{ n \geq 1 : D_{1}(n) \cdot (D_{2}(n))^{2} \leq  \varepsilon \mbox{ or }D_{2}(n) \cdot (D_{1}(n))^{2} \leq \varepsilon \}.$$
Notice that $\theta < \infty$ if and only of Adam (and Eve) is not in the $ \varepsilon$-packet for some $n$. Recall that conditionally on the past before $\theta$, the processes $(D_{1}(n) : n \geq \theta)$ and $(D_{2}(n) : n \geq \theta)$ are positive martingales starting respectively from $D_{1}(\theta)$ and $D_{2}(\theta)$. By Doob's inequality (for positive supermartingales), we have 
$$ \mathbb{P}\left( \sup_{n \geq \theta} D_{1}(n) \geq 3 \cdot D_{1}(\theta)\ \Bigg| \ {\theta}< \infty\right) \leq \frac{1}{3}.$$ Using a similar inequality for $D_{2}(\cdot)$ we deduce that conditionally on $\theta <\infty$, there is a probability at least $1- \frac{1}{3} - \frac{1}{3} = \frac{1}{3}$ that $D_{1}(n) \leq 3 D_{1}(\theta)$ and $D_{2}(n) \leq 3 D_{2}(\theta)$ for all $n \geq \theta$ and in particular (recall  \eqref{eq:degreelimite}) that $ \mathbf{D}_{1} \leq 3 D_{1}(\theta) $ and $ \mathbf{D}_{2} \leq 3  D_{2}(\theta)$. It follows that 
$$ \mathbb{P}( \theta < \infty) \cdot \frac{1}{3} \leq \mathbb{P}( \mathbf{D}_{1}  \mathbf{D}_{2}^{2} \leq 3^{3} \varepsilon \mbox{ or } \mathbf{D}_{2}  \mathbf{D}_{1}^{2} \leq 3^{3} \varepsilon).$$
Combined with the previous estimation \eqref{eq:limitfirst}, this completes the proof of Proposition~\ref{prop:adam_in_the_packet}.

\section{The $ \varepsilon$-packet has size $ \varepsilon^{-1+o(1)}$} \label{sec:size}
The  goal  of this section is to estimate the size of the $\epsilon$-packet $ \mathcal{P}_ \varepsilon(n)$ and prove the following proposition, which together with  Proposition~\ref{prop:adam_in_the_packet}, completes the proof of Theorem \ref{theo:main_result}.
 \begin{proposition}[Size of the $ \varepsilon$-packet] \label{prop:size_packet}
  Let $ \eta \in (0,1/8)$ be fixed. Then, for all $ \varepsilon >0$ small enough, we have
  \begin{eqnarray*} \mathbb{P} \left( \sup_{n \geq 1} |\mathcal{P}_{ \varepsilon}(n)| \leq \varepsilon^{-1- \eta} \right) \geq 1 - \varepsilon^{1- \eta}. \end{eqnarray*}
 \end{proposition}

For this purpose, we first establish uniform estimates on the renormalized degrees, so that heuristically $ D_i(n) \preceq \frac{1}{ \sqrt{i}}$ for all $n,i \geq 1$. Under this simplification, the $ \varepsilon$-packet is thus made of the pairs $(i,j)$ so that $  i \sqrt{j} \preceq \varepsilon^{-1}$ and $ \node{i} \sim \node{j}$. The number of such pairs is then easily estimated using concentration arguments.
 
\subsection{Uniform estimates on degrees}
Remember from \eqref{eq:defrenormalizeddegrees} and  \eqref{eq:degreelimite} that for $i \geq 1$, $d_{i}(n)$ denotes the degree of $\node{i}$ in $ \mathcal{T}(n)$ and $D_i(n)$ denotes its renormalized version which converges to $  \mathbf{D}_{i}$. We use again the explicit distribution of the limiting degrees derived by Pekös, Röllin and Ross:

\begin{lemma}[Another special case of Theorem 1.1 from \cite{pekoz2017joint}]\label{lem:lawdi} Suppose that $\emph{\node{i}}$ has degree $ m \in \{1,2,\ldots\}$ at time $k \geq i$. Then, its limit degree has the following explicit conditional distribution:
$$ \left(D_{i}(n) | d_{i}(k) = m \right) \xrightarrow[n\to\infty]{a.s.} (1-B_{(m, k)}) Z_{k} $$
where $ B_{(m, k)} \sim \beta(2(k-1)- m, m)$ and $Z_{k} \sim \mathrm{GG} (2k-1,2)$. Furthermore,  the variables $B_{(m, k)}$ and $Z_{k}$ are independent. In particular, for $k=i$ and $m=1$ we deduce that $ \mathbf{D}_i$ has the same distribution as $Z_i(1-B_{i-1})$ where $B_{i-1} \sim \mathrm{\beta} (2i-3,1)$ and $Z_i \sim \mathrm{GG}(2i-1,2)$. 
\end{lemma}
Note that $\expect{}{\bf{D}_i} \sim \frac{1}{2\sqrt{i}}$ as $ i \to \infty$. The following lemma roughly states that it is very unlikely that $D_{i}(n)$ deviates from above of  its typical order of magnitude $ \frac{1}{ \sqrt{i}}$, and this uniformly in $i,n \geq 1$.
\begin{lemma}[Upper deviations for the degrees]\label{lem:dev_degree} There exists $ \mathrm{Cst}>0$ so that for every $i \geq 1$ and all $A >8$ we have 

	$$\P\Big( \sup_{n \geq i}  D_i(n) \geq \frac{A}{\sqrt{i}}\Big)\leq  \mathrm{Cst}\cdot \exp\left( -\frac{A^{2/3}}{ \mathrm{Cst}}\right).$$
\end{lemma}
\begin{proof}
Let us first prove the limiting case 
 \begin{eqnarray} \label{eq:limitstep1}\P\Big(  \mathbf{D}_i \geq \frac{A}{\sqrt{i}}\Big)\leq  \mathrm{Cst}\cdot \exp( -A^{2/3}).  \end{eqnarray}
Indeed, with the notation of Lemma~(\ref{lem:lawdi}), we have
\begin{eqnarray*}\label{eq:control_asymptot_degree}\P\Big( \mathbf{D}_i \geq \frac{A}{\sqrt{i}}\Big) &=& \P\Big(Z_i(1-B_{i-1}) \geq \frac{A}{\sqrt{i}}\Big)\\
& \leq &  \P\Big(Z_i \geq A^{1/3} \sqrt{i}\Big) + \P\Big(Z_i(1-B_{i-1}) \geq \frac{A}{\sqrt{i}} \mbox{ and } Z_i \leq A^{1/3} \sqrt{i} \Big) \\
& \leq & \P\Big(Z_i \geq A^{1/3}  \sqrt{i}\Big) +\P\Big((1-B_{i-1}) \geq \frac{A^{2/3}}{i}\Big).
\end{eqnarray*}
Using the exact density \eqref{eq:densities} we can compute the second probability and bound it from above by    \begin{eqnarray} \label{eq:firstpart} \P\Big((1-B_{i-1}) \geq \frac{A^{2/3}}{i}\Big) = \Big(\big(1-\frac{A^{2/3}}{i}\big)_+\Big)^{2i-3} \leq  \mathrm{Cst} \cdot \exp(-A^{2/3})  \end{eqnarray} for some constant $ \mathrm{Cst}>0$. As for the first probability, we use the fact that if $x  \geq 2\sqrt{i}$, then the derivative of $x \mapsto -x^2 + (2i-2) \log(x)$ is smaller than $-4$. Thus, if  $x  \geq A^{1/3}\sqrt{i}$ with $A^{1/3} \geq 2$,
\begin{align*}x^{2i-2} \mathrm{e}^{-x^2} \leq (A^{1/3} \sqrt{i})^{2i-2} \mathrm{e}^{- A^{2/3} i} \cdot \mathrm{e}^{-4(x- A^{1/3} \sqrt{i})}.
\end{align*}
Using again the exact density \eqref{eq:densities}, and the fact that for all $i \geq 1$, we have $i^{i-1} \leq \mathrm{e}^i \Gamma \left( \frac{2i-1}{2}\right)/ \sqrt{2 \pi}$
\begin{align*}\P\Big(Z_i \geq A^{1/3} \sqrt{i}\Big) &\leq \int_{A^{1/3} \sqrt{i}}^{ + \infty} \mathrm{e}^{-4x} \mathrm{d}x \frac{2 (A^{2/3}i)^{i-1}}{ \Gamma \left( \frac{2i-1}{2}\right)} \mathrm{e}^{- A^{2/3}i+ 4 A^{1/3} \sqrt{i}} \\
&\leq \mathrm{Cst} \cdot  (\mathrm{e}^{- 4 A^{1/3} \sqrt{i}}) \cdot{A^{2(i-1)/3}  \mathrm{e}^{i}} \cdot   \mathrm{e}^{- A^{2/3}i+ 4 A^{1/3} \sqrt{i}} \\
& \leq \mathrm{Cst} \cdot  (A^{2/3}  \mathrm{e})^{i-1} \cdot  \mathrm{e}^{- A^{2/3}i} \\
& \leq  \mathrm{Cst} \cdot \exp \left( - A^{2/3}\right),
\end{align*}
which  together with \eqref{eq:firstpart} proves \eqref{eq:limitstep1}. \\
We now make the estimation uniform in $n$ using a stopping time argument similar as the one used in the proof of  Proposition \ref{prop:adam_in_the_packet}. We shall show that for $A > 8$ we have 
$$\P\left( \sup_{n \geq i}  D_i(n) \geq \frac{A}{\sqrt{i}}\right) \leq \mathrm{Cst} \cdot  \mathbb{P}\left(\mathbf{D}_i \geq A \cdot \frac{1}{4\sqrt{i}}\right), $$  for some   $ \mathrm{Cst}>0$. 
This,  together with \eqref{eq:control_asymptot_degree} gives the lemma. To do this, fix $A >8$ and let us introduce for $i \geq 1$ the stopping time $$ \theta_i = \inf \left\{ n \geq i :\ D_i(n) \geq  \frac{A}{\sqrt{i}} \right\}, $$ so that  $\P( \sup_{n \geq i}  D_i(n) \geq \frac{A}{\sqrt{i}})= \mathbb{P}(\theta_{i}< \infty)$. Thus, partitioning on the value of $\theta_i$ and applying the Markov property, we deduce that 
\begin{eqnarray*}
	 \mathbb{P}\left(\mathbf{D}_i \geq \frac{A}{4\sqrt{i}}\right) &\geq& \sum_{k=i}^{ + \infty} \mathbb{P}\left( {D}_i(k) \geq \frac{A}{\sqrt{i}}, \ \theta_{i} = k \right) \mathbb{P}\left.\left( \mathbf{D}_i \geq A \cdot \frac{1}{4\sqrt{i}} \right| {D}_i(k) \geq \frac{A}{\sqrt{i}}, \ \theta_{i} = k \right) \\
	 &\underset{\text{Markov}}{\geq}& \sum_{k=i}^{ + \infty} \mathbb{P}\left( {D}_i(k) \geq \frac{A}{\sqrt{i}}, \ \theta_{i} = k \right)\mathbb{P}\left.\left( \mathbf{D}_i \geq A \cdot \frac{1}{4\sqrt{i}} \right| {D}_i(k) \geq \frac{A}{\sqrt{i}}\right) \\
	&\geq& 
 \mathbb{P}(\theta_i< \infty) \cdot \inf \left\{ \mathbb{P}\left.\left(\mathbf{D}_i \geq \frac{A}{4\sqrt{i}} \right| d_i(k) =m\right) : \begin{array}{c} k \geq i, m \geq 1\\ 2(k-1)>m\\ m \geq  \frac{A}{ \sqrt{i}} \cdot ( \sqrt{\pi} \alpha_k) \end{array} \right\}
\end{eqnarray*}
Our goal is thus to provide a lower bound for the infimum on the right-hand side. Fix $ i \leq k$ and $m$ so that $d_i(k)=m$ happens with positive probability. Recalling Lemma \ref{lem:lawdi} we have
\begin{eqnarray*}
\mathbb{P}\left(\mathbf{D}_i \geq \frac{A}{4\sqrt{i}} \Big| d_i(k) =m\right) &=&
 \mathbb{P}\left(\left(1-B_{\left(m,k\right)}\right)Z_k \geq \frac{A}{4\sqrt{i}} \right). 
\end{eqnarray*}
Recall that $Z_{k} \sim \mathrm{GG} (2k-1,2)$ and $ B_{(m, k)} \sim \beta(2(k-1)- m, m)$ are independent. Notice that $1-B_{(m, k)} \sim \beta(m, 2(k-1)-m)$. Using the shorthand $X_{m,k} = (1-B_{\left(m,k\right)})Z_k$ then we have  $$ \mathbb{E}[X_{m,k}] = \frac{m}{2(k-1)} \cdot \frac{ \Gamma(k)}{  \Gamma \left( (2k-1)/2 \right) } \leq \sqrt{k}  \quad \mbox{ and }\quad  \mathbb{E}[X_{m,k}^{2}] =\frac{m (m+1)}{2(k-1)(2k-1)}\cdot \frac{ \Gamma((2k+1)/2)}{\Gamma((2k-1)/2)}.$$ We can then apply Paley-Zygmund inequality, since $2k-2 >  \frac{A}{ \sqrt{i}} \cdot ( \sqrt{\pi} \alpha_k)$ and get that \begin{align*}
 \mathbb{P}\left(X_{m,k} \geq \frac{A}{4\sqrt{i}} \right) &\geq  \mathbb{P}\left(X_{m,k} \geq A \cdot \frac{1}{4\sqrt{ik}}  \mathbb{E}\left[X_{m,k}k\right] \right) \\
 &\geq \left(1-\frac{A}{4 \sqrt{ik}}\right)^2 \frac{m}{m+1} \frac{(2k-1)\Gamma(k)^2}{2(k-1) \Gamma \left( (2k-1)/2 \right)\Gamma \left( (2k+1)/2 \right)} \\
 &\geq \frac{1}{2}\left(1-\frac{A}{2 \sqrt{ik}}\right)^2 \underset{A \leq \sqrt{ik}}{\geq} \frac{1}{8},
\end{align*}
which concludes the proof.
\end{proof}

\subsection{Size of the $\varepsilon$-packet}
We now have all the tools to prove Proposition~\ref{prop:size_packet}.
Fix $ \eta \in (0,1/8)$. We shall restrict to those small $  \varepsilon>0$ so that $\varepsilon^{-\eta} \geq - \mathrm{Cst} \log \varepsilon$ and $ \eta \varepsilon^{-\eta} > 2 \mathrm{Cst}$, where $ \mathrm{Cst}$ is the constant that appears in Lemma \ref{lem:dev_degree}. In this range of values, up to further decreasing $ \varepsilon$ to apply Lemma \ref{lem:dev_degree}, using $i ^\eta = \exp(\eta \log i) \geq 1+ \eta \log i$  we get 
\begin{align*} \mathbb{P} \left( \exists i \sup_{n \geq i} D_i(n) \geq \frac{(i/ \varepsilon)^{3\eta/2}}{ \sqrt{i}} \right) \underset{ \mathrm{Lem.} \ref{lem:dev_degree} }\leq \mathrm{Cst} \cdot\sum_{i \geq 1}   \exp \left(-  \frac{\varepsilon^{-\eta} i^\eta}{ \mathrm{Cst}}\right) &\leq \mathrm{Cst} \cdot \sum_{i \geq 1} \exp \left( -\frac{\varepsilon^{-\eta}}{ \mathrm{Cst}}(1+\eta \log i)\right)\\
& \leq  \mathrm{Cst} \cdot \sum_{i \geq 1}  \exp( \log \varepsilon - 2 \log i)\\  &= \mathrm{Cst} \cdot  \varepsilon \sum_{i \geq 1} \frac{1}{i^2}  \leq \mathrm{Cst} \cdot \varepsilon.
\end{align*}
Therefore, we can now work on the event of probability at least $1- \mathrm{Cst} \cdot \varepsilon$ defined by $$ \mathrm{Good}:=\left\{  \forall i \geq 1,\ \sup_{n \geq i} D_i(n) <  \frac{(i/ \varepsilon)^{3\eta/2}}{ \sqrt{i}} \right\}.$$
By considering only pairs $ i \leq j$ (up to a multiplicative factor 2) and upper bounding $D_i(n)$ and $D_j(n)$ by $(i/ \varepsilon)^{3\eta/2}/  \sqrt{i}$ and $ (j/ \varepsilon)^{3\eta/2}/ \sqrt{j}$ respectively, we deduce that 
\begin{align*}  \sup_{n  \geq 1}|  \mathcal{P}_{\epsilon} (n)| \mathbf{1}_{ \mathrm{Good}}  \leq & 2 \left|\left \{(i,j):  \ \begin{array}{c} 1 \leq  i <j , \\ \node{i} \sim \node{j} \end{array},  \quad \frac{(i/ \varepsilon)^{3\eta} (j/ \varepsilon)^{\frac{3\eta}{2}}}{ i\sqrt{j}} > \varepsilon \mbox{ or } \frac{(i/ \varepsilon)^{\frac{3\eta}{2}} (j/\varepsilon)^{ 3\eta}}{ \sqrt{i} j} > \varepsilon \right \}\right|  \mathbf{1}_{ \mathrm{Good}} \\
\leq & 2 \left|\left \{(i,j):  \ \begin{array}{c} 1 \leq  i <j, \\ \node{i} \sim \node{j} \end{array},  \quad  i \sqrt{j} \leq  \varepsilon^{-\frac{1+ 9\eta/2}{1- 3 \eta }}  \right \}\right| \mathbf{1}_{ \mathrm{Good}}
\end{align*}
since $ i \mapsto  i^{-1+3\eta/2}$ is decreasing. To simplify notation we put $-\frac{1+ 9 \eta/2}{1- 2  \eta } = -1-\eta'$ in the rest of the proof. 
Moreover, notice that if  $j > \varepsilon^{-2-2 \eta'}$, then  ${\sqrt{j}}   > \varepsilon^{-1- \eta'}$ and the inequality in the display above can not be satisfied; and recall that for all $j \geq 2$, there is a unique $i<j$ such that $ \node{i} \sim \node{j}$. Hence we can write 
\begin{align*} 
  \sup_{n \geq 1}| P_{\epsilon} (n)|  \mathbf{1}_{ \mathrm{Good}} &\leq 2 \sum_{1 \leq j \leq \varepsilon^{-2- 2 \eta'}} \underbrace{\mathbf{1}{ \left\{ \exists\ 1\leq i < j : \node{i} \sim \node{j},\  i \leq \frac{\epsilon^{-1- \eta'}}{\sqrt{j}} \right\}} }_{=:X_j}  \mathbf{1}_{ \mathrm{Good}}
 \end{align*}
%
Let us denote $ \mathcal{F}_j$ the natural filtration generated by the $(d_i(k) : 1 \leq i \leq k \leq j)$. Then conditionally on $ \mathcal{F}_{j-1}$, the random variable $X_j$ has a Bernoulli distribution with parameter $$ \mathrm{Param}(j) = \sum_{ 1 \leq i \leq \frac{\epsilon^{-1- \eta'}}{\sqrt{j}} \wedge j} \frac{d_i(j-1)}{2(j-1)}.$$
Notice then that on the event $ \mathrm{Good}$, conditionally on $ \mathcal{F}_{j-1}$,  the variable $\mathbf{1}_{ \mathrm{Good}} X_j$ is stochastically dominated by  a Bernoulli variable $Y_j$ independent of $ \mathcal{F}_{j-1}$ with a fixed parameter
\begin{align*} p_j &:= 1 \wedge \sum_{ 1 \leq i \leq \frac{\epsilon^{-1- \eta'}}{\sqrt{j}}} \frac{\sqrt{\pi}\alpha_j (i/ \varepsilon)^{3 \eta/2}}{ 2(j-1)\sqrt{i}} \leq 1 \wedge  \frac{ \mathrm{Cst}}{ \sqrt{j}} \varepsilon^{-3 \eta/2} \sum_{i=1}^{\frac{\epsilon^{-1- \eta'}}{\sqrt{j}}} i^{-1/2+3 \eta/2} \leq 1 \wedge \mathrm{Cst} \cdot  \left(\frac{\varepsilon^{-1/2}}{j^{3/4}}\right)^{1+ \eta''}, \end{align*} for some exponent $\eta''$ which tends to $0$ as $\eta \to 0$. 
When $j < \epsilon^{-5/6}$, we will simply bound from above the variable $Y_j$ by $1$, while when $j \geq  \epsilon^{-5/6}$, since $p_j < 1/2$ we shall dominate it stochastically by an independent Poisson random variable of parameter $2 p_j$. We deduce the following stochastic domination:
$$ \sup_{n \geq 1}|  \mathcal{P}_{\epsilon} (n)| \mathbf{1}_{ \mathrm{Good}} \underset{ \mathrm{stoch.}}{\leq}  2\left( \varepsilon^{-5/6} + 
\mathrm{Poisson} \left(   \mathrm{Cst} \cdot \sum_{ \varepsilon^{-5/6} \leq j \leq \varepsilon^{-2- 2 \eta'}} \left(\frac{\varepsilon^{-1/2}}{j^{3/4}}\right)^{1+ \eta''} \right)  \right).$$ The parameter of the Poisson random variable is bounded from above for small $ \varepsilon$ by $ \varepsilon^{-1- \eta'''}$ for yet another small constant $ \eta'''$ that goes to $0$ as $ \eta \to 0$. Using 
%
Benett's inequality \cite{taobennett}, it is very unlikely that such Poisson random variable take values much larger than $ \varepsilon^{-1- \eta'''}$ since 
\begin{align*}
\P \left(  \mathrm{Poisson}( \varepsilon^{-1-\eta'''}) \geq 2 \varepsilon^{-1-\eta'''}\right) \leq  \exp \left( -  \varepsilon^{-1-\eta'''} \cdot \left(2\log(2)-1\right) \right) \leq \epsilon,
\end{align*} for small $ \varepsilon>0$. Gathering-up the pieces, we see that with probability at least $1-  ( \mathrm{Cst}+1) \varepsilon$,  the size of  $ \mathcal{P}_{ \varepsilon} (n)$ is less than $ 6  \varepsilon^{-1-\eta'''}$ and this completes the proof, the section and the paper.

\bibliographystyle{alpha}

\bibliography{biblio}

\end{document}